\documentclass[twoside,12pt,a4paper]{amsart}
\usepackage{amscd,amsmath}
\usepackage{amssymb}
\usepackage{amsthm}
\usepackage{color}
\usepackage{hyperref}
\usepackage{enumerate}
\usepackage{geometry}
\theoremstyle{definition}
\newtheorem{theorem}{Theorem}[section]

\newtheorem{lemma}[theorem]{Lemma}
\newtheorem{corollary}[theorem]{Corollary}
\newtheorem{remark}[theorem]{Remark}

\numberwithin{equation}{section}
\newtheorem{problem}{Problem}

\newtheorem{acknowledgements}{Acknowledgement}

\def\<{\left < }
\def\>{\right >}
\def\({\left ( }
\def\){\right )}

\def\C2{${\bf C}^2$}

\begin{document}

\title[Gen. ineq. and shape op. ineq. of cont. CR-wp in cosympl. sp. form]{General inequalities and new shape operator inequality for contact CR-warped product submanifolds in cosymplectic space form} 

\author[A. Mustafa]{ABDULQADER MUSTAFA}
\address{Department of Mathematics, Faculty of Arts and Science, Palestine Technical University, Kadoorei, Tulkarm, Palestine}
\email{abdulqader.mustafa@ptuk.edu.ps}
\author[A. Assad]{ATA ASSAD}
\address{Department of Mathematics, Faculty of Arts and Science, Palestine Technical University, Kadoorei, Tulkarm, Palestine}
\email{a.asad@ptuk.edu.ps}
\author[C. \"Ozel]{CENAP \"OZEL}
\address{Department of Mathematics, Faculty of Science, King Abdulaziz University, 21589 Jeddah, Saudi Arabia}
\email{cozel@kau.edu.sa}
\author[A. Pigazzini]{ALEXANDER PIGAZZINI}
\address{Mathematical and Physical Science Foundation, 4200 Slagelse, Denmark}
\email{pigazzini@topositus.com}

\maketitle
\begin{abstract}
We establish two main inequalities; one for the norm of the second fundamental form and the other for the matrix of the shape operator. The results obtained are for cosymplectic manifolds and, for these, we show that the contact warped product submanifolds naturally possess a geometric property; namely $\mathcal{D}_1$-minimality which, by means of the Gauss equation, allows us to obtain an optimal general inequality. For sake of generalization, we state our hypotheses for nearly cosymplectic manifolds, then we obtain them as particular cases for cosymplectic manifolds. 
\\
For the other part of the paper, we derived some inequalities and applied them to construct and introduce a shape operator inequality for cosimpleptic manifolds involving the harmonic series.
\\
As further research directions, we have addressed a couple of open problems arose naturally during this work and which depend on its results.



\end{abstract}


\sloppy
\section{Introduction}
Warped product is a very important mathematical tool in the theory of general relativity. These mathematical structure are still widely studied today, because they can provide the best mathematical models of our universe, as for example the Robertson-Walker models, the Friedmann cosmological models, or the relativistic model of the Schwarzschild spacetime, which admits a warped product construction \cite{iijj77}.  

As seen in \cite{genIneq}, the aim of this paper continues in the search for the control of extrinsic quantities in relation to intrinsic quantities of Riemannan manifolds through the Nash theorem and consequent applications (\cite{2233ee}, \cite{55kk99}).
This motivation, led Chen to face the following problem:

\begin{problem}\label{prob3}
\cite{aallr4} Establish simple relationships between the main extrinsic invariants and the main intrinsic invariants of a submanifold.
\end{problem}
 
Several famous results in differential geometry, such as isoperimetric inequality, Chern-Lashof's inequality, and Gauss-Bonnet's theorem among others, can be regarded as results in this respect. The current paper aims to continue this sequel of inequalities.   
\\
This current paper is organized to include six sections. The section four, based on the results already obtained in \cite{genIneq}, states a general inequality involving the scalar curvature and the squared norm of the second fundamental form for a contact CR-warped product submanifold in a cosymplectic space form (i.e., \textit{Theorem 4.2}). In section five, we introduce a new type of inequalities for the shape operator matrix which involve the harmonic serie (\textit{Theorem 5.2} and \textit{Theorem 5.5}). In the final section, we assume that two open problems arose naturally due to the results of this work.

\section{Preliminaries}

As a preliminary we refer to what the authors have already explained in \cite{genIneq}(\textit{Section 2}).

\section{The Existence of $\mathcal{D}_1$-Minimal Contact Warped product Submanifolds in Cosymplectic  Manifolds}

Using the same calculations made in \cite{genIneq}, this sections aims to prove the existence of  $\mathcal{D}_1$-Minimal contact warped product submanifolds in cosymplectic  manifolds; namely, the contact $CR$-warped product submanifold of the type $M^n=N_T\times _fN_\perp$, where the characteristic vector field $\xi$ is tangent to $N_T$.

It is well-known that the contact $CR$-warped product submanifold of the type $M^n=N_\perp\times _fN_T$ in nearly cosymplectic manifolds is trivial warped product, reversing the factors then we consider the contact $CR$-warped product submanifold of the type $M^n=N_T\times _fN_\perp$ in nearly cosymplectic manifolds $\tilde {M}^{2l+1}$.

For the purpose of generalization, we consider the contact CR-warped product submanifolds in a nearly cosymplectic manifold.
The computations, already shown in \cite{genIneq} (i.e. \textit{Lemma 3.1, Lemma 3.3 and Corollary 3.2}), will not change for this case, and based on this, it is trivial to further extend the same results even for cosmplectic manifolds as a special case. Now, considering the above, as a particular case of the \textit{Lemma 4.1} (present in \cite{genIneq}),we finally can state the following new corollary:

\begin{corollary}\label{104}
The contact $CR$-warped product submanifold of  type $M^n=N_T\times _fN_\perp$ is $\mathcal{D}_T$-minimal in cosymplectic  manifolds.
\end{corollary}

\section{Special Inequalities and Applications}

Considering \cite{genIneq}(\textit{Theorem 6.1}), we obtain:

\begin{theorem}\label{299}
Let $\varphi :M^n=N_1\times _fN_2 \longrightarrow \tilde M^{2m+1}(c_c)$ be a $\mathcal {D}_1$-minimal isometric immersion of a contact warped product submanifold $M^n$ into a cosymplectic space form $\tilde {M}(c_c)^{2m+1}$. Then, we have
 \begin{enumerate}
\item[(i)]$||h||^2\ge 2n_2 \biggl(||\nabla (\ln f)||^2-\Delta (\ln f)+(n_1-1)\frac{c_{c}}{4}\biggr).$
\item[(ii)] The equality in (i) holds identically if and only if $N_1$, $N_2$ and $M^n$ are totally geodesic, totally umbilical and minimal submanifolds in $\tilde M^{2m+1}(c_c)$, respectively.
\end{enumerate}
\end{theorem}

Since contact $CR$-warped product submanifolds are $\mathcal{D}_1$-minimal in a cosymplectic space form $\tilde {M}(c_c)^{2m+1}$, then:

\begin{theorem}\label{299}
Let $\varphi :M^n=N_T\times _fN_\perp \longrightarrow \tilde M^{2m+1}(c_c)$ be an  isometric immersion of a contact $CR$-warped product submanifold $M^n$ into a cosymplectic space form $\tilde {M}(c_c)^{2m+1}$. Then, we have
 \begin{enumerate}
\item[(i)]$||h||^2\ge 2n_2 \biggl(||\nabla (\ln f)||^2-\Delta (\ln f)+(n_1-1)\frac{c_{c}}{4}\biggr).$
\item[(ii)] The equality in (i) holds identically if and only if $N_T$, $N_\perp$ and $M^n$ are totally geodesic, totally umbilical and minimal submanifolds in $\tilde M^{2m+1}(c_c)$, respectively.
\end{enumerate}
\end{theorem}

\begin{remark}
In Euclidean cosymplectic space forms, part (i) of the above two inequalities reduces to

$$||h||^2\ge 2n_2 \biggl(||\nabla (\ln f)||^2-\Delta (\ln f)\biggr).$$
\end{remark}

\section{Inequalities of the Shape Operator Matrix and harmonic serie}

In this section we consider two important relations, the first is:
$$(I)\; \;  h(\xi, \xi)=0,$$ 
in fact, in our \textit{Section 3}, we specify that $\xi$ is tangent to the first factor and from  \cite{genIneq}(\textit{Corollary 3.2}(i)) (which is also trivialy valid for contact CR- warped product submanifold in a nearly cosymplectic manifold), we obtain $(I)$, while the second relation is:
 $$(II)\; \; <A_\zeta X, Y>=<h(X,Y), \zeta>,$$
where $A$ and $h$ are the shape operator and the second fundamental form respectively. 
\\
We consider the matrix of the shape operator putting $v=n-1$. Therefore, in the rest of this paper we demonstrate some geometric and arithmetic inequalities for such matrices of order $v\times v$.
\\
Let $%
\mathbb{M}%
_{v}(%
\mathbb{C}
)$ be the algebra of all $v\times v$ complex matrices. The singular values
$t_{1}(A),...,t_{v}(A)$ of a matrix $A\in%
\mathbb{M}%
_{v}(%
\mathbb{C}
)$ are the eigenvalues of the matrix $\left(  A^{\ast}A\right)  ^{1/2}$
arranged in decreasing order and repeated according to multiplicity. A
Hermitian matrix $A\in%
\mathbb{M}%
_{v}(%
\mathbb{C}
)$ is said to be positive semidefinite, written as $A\geq0$, if $x^{\ast
}Ax\geq0$ for all $x\in%
\mathbb{C}
^{v}$ and it is called positive definite, written as $A>0$, if $x^{\ast}Ax>0$
 for all $x\in%
\mathbb{C}
^{v}$ with $x\neq0$. The Hilbert-Schmidt norm (or the Frobenius norm)
$\left\Vert \cdot\right\Vert _{2}$\ is the norm defined on $%
\mathbb{M}%
_{v}(%
\mathbb{C}
)$\ by $\left\Vert A\right\Vert _{2}=\left(
{\displaystyle\sum\limits_{j=1}^{v}}
t_{j}^{2}(A)\right)  ^{1/2}$, $A\in%
\mathbb{M}%
_{v}(%
\mathbb{C}
)$. The Hilbert-Schmidt norm is unitarily invariant, that is $\left\Vert
UAV\right\Vert _{2}=\left\Vert A\right\Vert _{2}$ for all $A\in%
\mathbb{M}%
_{v}(%
\mathbb{C}
)$ and all unitary matrices $U,V\in%
\mathbb{M}%
_{v}(%
\mathbb{C}
)$. Another property of the Hilbert-Schmidt norm is that $\left\Vert
A\right\Vert _{2}=\left(
{\displaystyle\sum\limits_{i,j=1}^{v}}
\left\vert a_{j}^{\ast}Ab_{i}\right\vert ^{2}\right)  ^{1/2}$, where
$\{b_{j}\}_{j=1}^{v}$ and $\{a_{j}\}_{j=1}^{v}$ are two orthonormal bases of $%
\mathbb{C}
^{v}$. The spectral matrix norm, denoted by $\left\Vert \cdot\right\Vert $, of
a matrix $A\in%
\mathbb{M}%
_{v}(%
\mathbb{C}
)$ is the norm defined by $\left\Vert A\right\Vert =\sup\{\left\Vert
Ax\right\Vert :x\in%
\mathbb{C}
^{v},\left\Vert x\right\Vert =1\}$ or equivalently $\left\Vert A\right\Vert
=t_{1}\left(  A\right)  $, For further properties of these norms the reader is
referred to \cite{44} or \cite{33}. A matrix $A\in%
\mathbb{M}%
_{v}(%
\mathbb{C}
)$ is called contraction if $\left\Vert A\right\Vert \leq1$, or equivalently,
$A^{\ast}A\leq I_{v}$, where $I_{v}$ is the identity matrix in $%
\mathbb{M}%
_{v}(%
\mathbb{C}
)$.

An $v\times v$ $A=(\alpha_{ij})$ is called dobly stochastic if and only if
$\alpha_{ij}\geq0,$ for all $i,j=1,...,v,$ summation all column and row enteries equal
one \cite{55}.

Let $%
\mathbb{M}%
_{v}(%
\mathbb{C}
)$ be the algebra of all $v\times v$ complex matrices. For a matrix $A\in%
\mathbb{M}%
_{v}(%
\mathbb{C}
)$, let $\lambda_{1}(A),...,\lambda_{v}(A)$ be the eigenvalues of $A$ repeated
according to multiplicity. The singular values of $A$, denoted by
$t_{1}(A),...,t_{v}(A)$, are the eigenvalues of the matrix $\left\vert
A\right\vert =\left(  A^{\ast}A\right)  ^{1/2}$ arranged in decreasing order
and repeated according to multiplicity. A Hermitian matrix $A\in%
\mathbb{M}%
_{v}(%
\mathbb{C}
)$ is said to be positive semidefinite if $x^{\ast}Ax\geq0$ for all $x\in%
\mathbb{C}
^{v}$ and it is called positive definite if $x^{\ast}Ax>0$ for all $x\in%
\mathbb{C}
^{v}$ with $x\neq0$. The direct sum of matrices \\ \\ $A_{1},...,A_{u}\in%
\mathbb{M}%
_{v}(%
\mathbb{C}
)$ is the matrix $\oplus_{i=1}^{u}A_{i}=\left[
\begin{array}
[c]{cccc}%
A_{1} & 0 & \cdots & 0\\
0 & A_{2} & \ddots & \vdots\\
\vdots & \ddots & \ddots & 0\\
0 & \cdots & 0 & A_{u}%
\end{array}
\right]  $. 
\\
\\
\\
For two matrices $A_{1},A_{2}\in%
\mathbb{M}%
_{v}(%
\mathbb{C}
),$ we write $A\oplus B$ instead of $\oplus_{i=1}^{2}A_{i}$.

The usual matrix norm $\left\Vert \cdot\right\Vert ,$ the Schatten $p$-norm
($p\geq1$), and the Ky Fan $k$-norms $\left\Vert \cdot\right\Vert _{(k)}$
$\left(  k=1,...,v\right)  $ are the norms defined on $%
\mathbb{M}%
_{v}(%
\mathbb{C}
)$ by $\left\Vert A\right\Vert =\sup\{\left\Vert Ax\right\Vert :x\in%
\mathbb{C}
,\left\Vert x\right\Vert =1\}$, $\left\Vert A\right\Vert _{p}=\sum_{j=1}%
^{v}t_{j}^{p}\left(  A\right)  $, and $\left\Vert A\right\Vert _{(k)}%
=\sum_{j=1}^{k}t_{j}\left(  A\right)  ,$ $k=1,...,v$. It is known that (see,
e.g., \cite[p. 76]{55}) for every $A\in%
\mathbb{M}%
_{v}(%
\mathbb{C}
)$ we have%
\begin{equation}
\left\Vert A\right\Vert =t_{1}\left(  A\right) \label{id1}%
\end{equation}
and for each $k=1,...,v,$ we have%
\begin{equation}
\left\Vert A\right\Vert _{(k)}=\max\left\vert \sum_{j=1}^{k}y_{j}^{\ast}%
Ax_{j}\right\vert ,\label{id0}%
\end{equation}
where the maximum is taken over all choices of orthonormal $k$-tuples
$x_{1},...,x_{k}$ and $y_{1},...,y_{k}$. In fact, replacing each $y_{j}$ by
$z_{j}y_{j}$ for some suitable complex number $z_{j}$ of modulus $1$ for which
$\bar{z}_{j}y_{j}^{\ast}Ax_{j}=\left\vert y_{j}^{\ast}Ax_{j}\right\vert $,
implies that the $k$-tuple $z_{1}y_{1},...,z_{k}y_{k}$ is still orthonormal,
and so an identity equivalent the identity (\ref{id0}) can be seen as follows:%
\begin{equation}
\left\Vert A\right\Vert _{(k)}=\max\sum_{j=1}^{k}\left\vert y_{j}^{\ast}%
Ax_{j}\right\vert ,\label{id2}%
\end{equation}
where the maximum is taken over all choices of orthonormal $k$-tuples
$x_{1},...,x_{k}$ and $y_{1},...,y_{k}$.

A unitarily invariant norm $\left\vert \left\vert \left\vert \cdot\right\vert
\right\vert \right\vert $ is a norm defined on $%
\mathbb{M}%
_{v}(%
\mathbb{C}
)$\ that satisfies the invariance property $\left\vert \left\vert \left\vert
UAV\right\vert \right\vert \right\vert =\left\vert \left\vert \left\vert
A\right\vert \right\vert \right\vert $ for every $A\in%
\mathbb{M}%
_{v}(%
\mathbb{C}
)$ and every unitary matrices $U,V\in%
\mathbb{M}%
_{v}(%
\mathbb{C}
)$. It is known that%
\[
\left\vert \left\vert \left\vert A\oplus A\right\vert \right\vert \right\vert
\geq\left\vert \left\vert \left\vert B\oplus B\right\vert \right\vert
\right\vert \text{ \ \ for every unitarily invariant norm}%
\]
if and only if%
\[
\left\vert \left\vert \left\vert A\right\vert \right\vert \right\vert
\geq\left\vert \left\vert \left\vert B\right\vert \right\vert \right\vert
\text{ \ \ for every unitarily invariant norm.}%
\]
Also,%
\[
\left\vert \left\vert \left\vert A\oplus B\right\vert \right\vert \right\vert
=\left\vert \left\vert \left\vert B\oplus A\right\vert \right\vert \right\vert
=\left\vert \left\vert \left\vert \left[
\begin{array}
[c]{cc}%
0 & B\\
A^{\ast} & 0
\end{array}
\right]  \right\vert \right\vert \right\vert.
\]

Hermitian matrix in real space is symetric matrix, so if $A,B$ are hermitian
matrices and $A-B$ is positive then we say that $B\leq A$ .Wely's monotoniciy
theorem say that the relation imply $\lambda_{j}(B)\leq\lambda_{j}(A)$ for all
$j=1,...,v$.\cite{55}
\\

Let $A,B$ be symmetric matrices $\pm A\leq B. $\; Then $t_{j}(A)\leq t_{j}(B\oplus
B)$. This lead to: $t_{j}(AB+BA)\leq t_{j}((A^{2}+B^{2}%
)\oplus(A^{2}+B^{2}))$ \cite{55}. In the following we construct inequalities of harmonic series.

Now let consider harmonic series in the following form $%
{\displaystyle\sum\limits_{v=1}^{\infty}}
\frac{1}{v}$ or in general form ,$%
{\displaystyle\sum\limits_{v=1}^{\infty}}
\frac{1}{\alpha\text{ }v\text{ }+d}$ where $\alpha\neq0,d$ are real numbers, $\frac
{\alpha}{d}$ is positive, and a generalization of the harmonic series is the
$p-\operatorname{series} $ (or hyperharmonic series), defined as $%
{\displaystyle\sum\limits_{v=1}^{\infty}}
\frac{1}{v^{p}}.$We have the following inequality (from \cite{11}, p. 202), respect to harmonic series.

\begin{lemma}\label{X}
\label{L0001} \ Let $v>1$, be positive integer then
\end{lemma}%

\begin{equation}
2\sqrt{v+1}-2\lessdot%
{\displaystyle\sum\limits_{k=1}^{v}}
\frac{1}{\sqrt{k}}<2\sqrt{v}-1\label{01}%
\end{equation}

\bigskip

so we can see that
\begin{align}%
{\displaystyle\sum\limits_{k=1}^{v}}
\sqrt{k}  & =%
{\displaystyle\sum\limits_{k=1}^{v}}
\frac{k}{\sqrt{k}}\label{1}\\
& \leq\left(
{\displaystyle\sum\limits_{k=1}^{v}}
k\right)  \left(
{\displaystyle\sum\limits_{k=1}^{v}}
\frac{1}{\sqrt{k}}\right) \nonumber\\
& \leq\frac{v(v+1)}{2})\left(  2\sqrt{v}-1\right) \nonumber\\
& =v(v+1)(\sqrt{v}-0.5)\nonumber
\end{align}
more general representation for it by the following theorem.

\begin{theorem}
\label{T010}Let $A\in%
\mathbb{M}%
_{v}(%
\mathbb{C}
)$ be positive definite matrix then%
\[
\left\Vert
{\displaystyle\sum\limits_{k=1}^{v}}
\sqrt{k}A^{k}\right\Vert _{2}<\frac{v(v+1)(\sqrt{v}-0.5)(\left\Vert
A\right\Vert _{2}-\left(  \left\Vert A\right\Vert _{2}\right)  ^{v+1}%
)}{1-\left\Vert A\right\Vert _{2}}%
\]

\end{theorem}

\begin{proof}
Let \ $A$ \ has singular values $\geq$ $t_{1}(A)\geq...\geq t_{v}(A)$ \ \ and
$U$ be a unitariy matix such that $A=Udig(t_{1}(A),...,t_{v}(A))U^{\ast}$ then.%
\begin{align*}
\left\Vert
{\displaystyle\sum\limits_{k=1}^{v}}
\sqrt{k}A^{k}\right\Vert _{2}  & =\left\Vert
{\displaystyle\sum\limits_{k=1}^{v-}}
dig(\sqrt{k}t_{1}^{k}(A),...,\sqrt{k}t_{v}^{k}(A))\right\Vert _{2}\\
& <%
{\displaystyle\sum\limits_{k=1}^{v-}}
\sqrt{k}\left\Vert dig(t_{1}^{k}(A),...,t_{v}^{k}(A))\right\Vert _{2}\\
& \leq v(v+1)(\sqrt{v}-0.5)%
{\displaystyle\sum\limits_{k=1}^{v-}}
\left\Vert dig(t_{1}^{k}(A),...,t_{v}^{k}(A))\right\Vert _{2}\\
& \leq v(v+1)(\sqrt{v}-0.5)%
{\displaystyle\sum\limits_{k=1}^{v-}}
(\left\Vert dig(t_{1}(A),...,t_{v}(A))\right\Vert _{2})^{k}\\
& =v(v+1)(\sqrt{v}-0.5)%
{\displaystyle\sum\limits_{k=1}^{v-}}
(\left\Vert A\right\Vert _{2})^{k}%
\index{1}%
\\
& =\frac{v(v+1)(\sqrt{v}-0.5)(\left\Vert A\right\Vert _{2}-\left\Vert
A\right\Vert _{2}{}^{v+1}}{1-\left\Vert A\right\Vert _{2}}%
\end{align*}

\end{proof}

We can have a result depend at this theorem for dobuly stochastic
matrices.\label{index1} we get the result imediatley.

\begin{corollary}
\label{C1}Le $A$ be positive definite doubly stochastic matrix then
\[
\left\Vert
{\displaystyle\sum\limits_{k=1}^{v}}
\sqrt{k}A^{k}\right\Vert _{2}\leq v^{2}(v+1)(\sqrt{v}-0.5)
\]

\end{corollary}

\begin{proof}
since $A$ be positive definite doubly stochastic matrix, we obtain \\ $0<\left\Vert
A\right\Vert _{2}\leq1$
\end{proof}

\begin{lemma}
Let $v>1$, be positive integer and $x_{k}$, $k=1,2,....v$, be positive number
then%
\begin{equation}
\frac{\min(x_{k})_{1\leq k\leq v}}{v(v+1)(\sqrt{v}-0.5)}\lessdot%
{\displaystyle\sum\limits_{k=1}^{v}}
\frac{x_{k}}{\sqrt{k}}<v(2\sqrt{v}-1)\max(x_{k})_{1\leq k\leq v}\label{2}%
\end{equation}

\end{lemma}

\begin{proof}
Since%
\begin{align*}%
{\displaystyle\sum\limits_{k=1}^{v}}
\frac{x_{k}}{\sqrt{k}}  & \leq\left(
{\displaystyle\sum\limits_{k=1}^{v}}
x_{k}\right)  \left(
{\displaystyle\sum\limits_{k=1}^{v}}
\frac{1}{\sqrt{k}}\right) \\
& <v(2\sqrt{v}-1)\max(x_{k})_{1\leq k\leq v}%
\text{\ \ \ \ \ \ \ \ \ \ \ \ \ (by the inequality (\ref{01}))}%
\end{align*}

The left side of the inequlity we can have%
\begin{align}%
{\displaystyle\sum\limits_{k=1}^{v}}
\frac{x_{k}}{\sqrt{k}}  & \geq\frac{%
{\displaystyle\sum\limits_{k=1}^{v}}
x_{k}}{%
{\displaystyle\sum\limits_{k=1}^{v}}
\sqrt{k}}\label{3}\\
& >\frac{v\min(x_{k})_{1\leq k\leq v}}{v(v+1)(\sqrt{v}-0.5)}\text{ \ (by the
inequality (\ref{1}))}\nonumber\\
& =\frac{\min(x_{k})_{1\leq k\leq v}}{(v+1)(\sqrt{v}-0.5)}\nonumber
\end{align}

\end{proof}

Based on the inequality (\ref{2}) we have the following Theorem

\begin{theorem}
\label{T0}Let $A,X\in%
\mathbb{M}%
_{v}(%
\mathbb{C}
)$, be positive definite matrice and let $\left\lceil t_{k}(A)\right\rceil
,\left\lfloor t_{k}(A)\right\rfloor \leq v$. Then%
\[
t_{v}(X)(2\sqrt{\left\lfloor t_{k}(A)\right\rfloor +1}-2)\lessdot%
{\displaystyle\sum\limits_{k=1}^{\lfloor t_{k}(A)\rfloor}}
t_{k}(X\text{ }A^{-0.5})<(2\sqrt{\left\lceil t_{k}(A)\right\rceil }-1)t_{1}(X)
\]

\end{theorem}

\begin{proof}
Since
\begin{align*}%
{\displaystyle\sum\limits_{k=1}^{\lfloor t_{k}(A)\rfloor}}
t_{k}(X\text{ }A^{-0.5})  & \leq%
{\displaystyle\sum\limits_{k=1}^{\lfloor t_{k}(A)\rfloor}}
t_{1}(X\text{ })t_{k}(A^{-0.5})\\
& =%
{\displaystyle\sum\limits_{k=1}^{\lfloor t_{k}(A)\rfloor}}
t_{1}(X\text{ })\left\lceil t_{k}(A)\right\rceil ^{-0.5}\\
& <t_{1}(X\text{ })(2\sqrt{\left\lceil t_{k}(A)\right\rceil }-1)\text{ \ by
lemma } \ref{X}%
\end{align*}
The left side inequality, since%
\begin{align*}%
{\displaystyle\sum\limits_{k=1}^{\lfloor t_{k}(A)\rfloor}}
t_{k}(X\text{ }A^{-0.5})  & \geq%
{\displaystyle\sum\limits_{k=1}^{\lfloor t_{k}(A)\rfloor}}
t_{v}(X)t_{k}(A^{-0.5})\\
& =%
{\displaystyle\sum\limits_{k=1}^{\lfloor t_{k}(A)\rfloor}}
t_{v}(X\text{ })\left\lfloor t_{k}(A)\right\rfloor ^{-0.5}\\
& >t_{v}(X)(2\sqrt{\left\lfloor t_{k}(A)\right\rfloor +1}-2)\text{\ by
lemma } \ref{X}%
\end{align*}

\end{proof}

\section{Research problems based on First Chen inequality}

Due to the results of this paper, we hypothesize a pair of open problems.

Firstly, we suggest the following:
\begin{problem}\label{ama1}
Prove the inequalities of the second fundamental form for $CR$-submanifolds in generalized Sasakian space forms.
\end{problem}

Secondly, we ask:
\begin{problem}\label{pqm2}
Prove the rest of inequalities of this paper for the matrix of the second fundamental form obtained in previous problem.
\end{problem}

\vskip.15in
\begin{acknowledgements}
The  authors  would like to thank the Palestine Technical University Kadoori, PTUK, for its supports to accomplish this work.
\end{acknowledgements}

\end{document}